\documentclass[11pt]{amsart}
\usepackage{amsmath,amssymb,amsfonts,amsthm,mathtools}
\usepackage[margin=1.05in]{geometry}
\usepackage[hidelinks]{hyperref}

\numberwithin{equation}{section}
\newtheorem{theorem}{Theorem}[section]
\newtheorem{corollary}[theorem]{Corollary}

\newtheorem{lemma}[theorem]{Lemma}
\theoremstyle{remark}
\newtheorem{remark}[theorem]{Remark}

\newcommand{\R}{\mathbb R}
\newcommand{\pcap}{\operatorname{Cap}}

\title[Outward minimizing $p$-capacity]
{Outward minimizing $p$-capacity, horizons,\\ and Schwarzschild rigidity}
\author{Sehong Park}
\address{Human-Centered Artificial Intelligence Research Institute,
Ewha Womans University, Seoul, Republic of Korea}
\email{pshong21@ewha.ac.kr}
\date{}

\begin{document}
\raggedbottom

\begin{abstract}
Let $(M^3,g)$ be a complete Riemannian manifold diffeomorphic to
$\R^3\setminus\{0\}$, with nonnegative scalar curvature.
Assume that a distinguished end is asymptotically flat,
with ADM mass $m_+$.
For each $p\in(1,3)$, define $c_{O,p}$ as the infimum of the
Schwarzschild-normalized $p$-capacity over outward-minimizing
finite-perimeter boundaries separating the two ends.
We prove that $m_+\ge c_{O,p}$ whenever $A(g)>0$, where $A(g)$
is the infimum of the areas of boundaries separating the two ends.  Equality
at a single exponent
produces a least-area horizon, forces its exterior to be the Schwarzschild
exterior of mass $m_+$, and yields equality at every exponent.  In the
equality case, if the second end is also asymptotically flat, its mass
satisfies $m_-\ge m_+$, with equality precisely for the two-sided spatial
Schwarzschild manifold.  We also show that a strict gap between $c_{O,p}$
and the unconstrained capacity infimum $c_{M,p}$ detects a horizon.
\end{abstract}

\maketitle

\section{Introduction}

Let $(M,g)$ be a complete orientable Riemannian $3$-manifold diffeomorphic
to $\R^3\setminus\{0\}$, with nonnegative scalar curvature $R_g\ge0$.
Denote its two ends by $e_+$ and $e_-$, and assume that $e_+$ is
asymptotically flat: in an asymptotically flat chart,
\[
 g_{ij}-\delta_{ij}=O_2(|x|^{-\tau}), \quad \tau>\tfrac12,
 \qquad R_g\in L^1,
\]
so that the ADM mass $m_+$ of $e_+$ \cite{ADM} is well defined and
chart-independent \cite{Bartnik86}.

The positive mass theorem of Schoen--Yau \cite{SchoenYau79} asserts
that on complete asymptotically flat manifolds without boundary the mass
is nonnegative, with equality only for the Euclidean space.  In the
distinguished-end setting used here, nonnegativity remains valid when
the other ends are arbitrary \cite{LesourdUngerYau, LLUIncomplete}.
When the manifold has an outermost minimal boundary, modeling an
apparent horizon, the Riemannian Penrose inequality
$m\ge\sqrt{|\partial M|/16\pi}$ was proved by Huisken--Ilmanen
\cite{HuiskenIlmanen} in the connected case via weak inverse mean curvature flow and by Bray
\cite{Bray} in general via a conformal flow of metrics; in each case, the
spatial Schwarzschild manifold is the unique case of equality.

The capacity enters the Penrose inequality through Bray's proof: there,
the monotonicity of the mass along the conformal flow rests on the
mass--capacity inequality
\begin{equation}\label{eq:bray}
 m\ \ge\ c_\Sigma
\end{equation}
for a minimal boundary $\Sigma$, with equality only for Schwarzschild
exteriors.  Bray--Miao \cite{BrayMiao} obtained a sharp capacity upper
bound in terms of area and Willmore energy, and Xiao \cite{Xiao}
extended both estimates to the $p$-capacity, $1<p<3$, under a
nonnegative Hawking mass assumption.  In the harmonic case $p=2$,
Miao \cite{MiaoPMJ} proved the mass--capacity estimate
without the condition $\int_\Sigma H^2\le16\pi$:
\begin{equation}\label{eq:miaoratio}
 \frac{m}{c_\Sigma}\ \ge\ 1-\Big(\frac1{16\pi}\int_\Sigma H^2\Big)^{1/2}
\end{equation}
for a connected boundary $\Sigma$ with $H_2(M,\Sigma)=0$; see also
\cite{HMT} for $p$-harmonic potentials.  Recently Xia--Yin--Zhou
\cite{XYZ} found monotone quantities for $p$-capacitary functions
calibrated on Schwarzschild exteriors, generalizing Miao's inequalities
to all $p\in(1,3)$; the mass-to-$p$-capacity inequality for minimal
boundaries was also obtained by Mazurowski--Yao \cite{MazurowskiYao}.

The minimal-boundary case of these inequalities assumes the existence
of a horizon. A related question is which geometric conditions force
a horizon to exist.  Schoen--Yau \cite{SchoenYau83} showed that sufficiently condensed matter
forces an apparent horizon, and Shi--Tam \cite{ShiTam07} detected
horizons in compact manifolds with boundary through conditions on
quasi-local mass.  For complete noncompact fill-ins of Bartnik data
\cite{LLU}, horizon existence criteria comparing quasi-local masses
with an asymptotic area threshold were obtained in joint work with Miao
\cite{MiaoPark}. 

We establish a mass--capacity inequality without assuming a horizon
and characterize its equality case.
We also show that a gap between two global capacities forces
a horizon to exist.

Zhu \cite{Zhu} used Gromov's $\mu$-bubble method
\cite{GromovFour} to prove a mass--systole inequality
and characterize its equality case.
When the infimum of separating areas is positive, his construction
produces strictly outward-minimizing spheres with uniformly bounded
areas and mean curvatures converging uniformly to zero
\cite[Proposition 2.2]{Zhu}.

Miao \cite{MiaoJGA} introduced a global harmonic capacity,
$\mathfrak{c}(M,g)$, by taking an infimum over inner boundaries.
He proved $m\ge 2\mathfrak{c}(M,g)$ under assumptions that include
a Ricci lower bound.
Bi--Zhu \cite{BiZhu} established the corresponding inequality
for spin manifolds with arbitrary ends, without the Ricci assumption.
Miao \cite{MiaoMono} also obtained Schwarzschild rigidity results
from monotonicity formulas for harmonic functions.

\subsection*{Setting and main results}
Let $\mathcal S$ be the collection of all boundaries
$\Sigma=\partial D_\Sigma$ of finite-perimeter regions
$D_\Sigma\subset M$ separating the two ends:
$D_\Sigma$ contains an end-neighborhood of $e_-$, and its
complement contains an end-neighborhood of $e_+$.
A smooth minimal member of $\mathcal S$ is called a \emph{horizon}.
Throughout, inclusions between finite-perimeter regions are understood
up to sets of volume zero.

Let $P(D)$ denote the perimeter of $D$ with respect to $g$.
For $\Sigma\in\mathcal S$, write
\[
 |\Sigma|:=P(D_\Sigma).
\]
Define
\begin{equation}\label{eq:area-infimum}
 A(g):=\inf_{\Sigma\in\mathcal S}|\Sigma|.
\end{equation}

For $\Sigma\in\mathcal S$ and $1<p<3$, define its exterior $p$-capacity
by
\[
 \pcap_p(\Sigma):=\inf_f\int_M|\nabla f|^p\,dV_g,
\]
where the infimum is over $f\in W^{1,p}_{\mathrm{loc}}(M)$ of finite
$p$-energy such that $f=0$ almost everywhere on $D_\Sigma$ and $f\to1$
at $e_+$.  Put
\[
 a=a_p:=\frac{3-p}{p-1},
 \qquad
 I_a(k):=\int_0^k s^{a-1}(1+s)^{-2a}\,ds \quad (0<k\le 1),
\]
and define the \emph{Schwarzschild-normalized $p$-capacity}
\begin{equation}\label{eq:normalized-capacity}
 c_{\Sigma,p}:=2\,I_a(1)^{1/a}
 \left(\frac{\pcap_p(\Sigma)}{4\pi}\right)^{1/(3-p)}.
\end{equation}
By the explicit computation of the $p$-capacitary potential on
Schwarzschild exteriors \cite{XYZ}, at the horizon of the
spatial Schwarzschild manifold of mass $m$ one has $c_{\Sigma,p}=m$ for
\emph{every} $p\in(1,3)$; for $p=2$, since $a=1$ and $I_1(1)=\tfrac12$,
the quantity \eqref{eq:normalized-capacity} is the usual normalized
harmonic capacity $c_\Sigma=\pcap_2(\Sigma)/4\pi$.  We also note
$I_a(1)=\tfrac12\mathcal B(a,a)$ \cite{XYZ}.

A boundary $\Sigma\in\mathcal S$ is \emph{outward minimizing} (toward
$e_+$) if
\[
 |\Sigma|\le P(F)
\]
for every finite-perimeter region $F\supset D_\Sigma$ with
$F\setminus D_\Sigma\Subset M$.  Let $\mathcal S_O\subset\mathcal S$
denote this subclass. Set
\begin{equation}\label{eq:global-capacities}
 c_{M,p}:=\inf_{\Sigma\in\mathcal S}c_{\Sigma,p},
 \qquad
 c_{O,p}:=\inf_{\Sigma\in\mathcal S_O}c_{\Sigma,p},
\end{equation}
so that trivially $c_{M,p}\le c_{O,p}$.  Exterior $p$-capacity is
monotone:
\begin{equation}\label{eq:capacity-monotonicity}
 D_{\Sigma_1}\subset D_{\Sigma_2}
 \ \Longrightarrow\ c_{\Sigma_1,p}\le c_{\Sigma_2,p}.
\end{equation}
See Appendix \ref{app:lemmas} for the strict form used below.

\begin{theorem}[Mass--capacity inequality and rigidity]\label{thm:main}
Assume $R_g\ge0$ and $A(g)>0$.  Then, for every $p\in(1,3)$,
\begin{equation}\label{eq:main-inequality}
 m_+\ \ge\ c_{O,p}.
\end{equation}
Equality holds for some $p_0\in(1,3)$ if and only if there is a horizon
$\Sigma_*\in\mathcal S_O$ such that
\[
 |\Sigma_*|=A(g)=16\pi m_+^2,
\]
the $e_+$-side of $\Sigma_*$ is the Schwarzschild exterior of mass
$m_+$, and every $\Gamma\in\mathcal S_O$ encloses $\Sigma_*$.  In that
case,
\[
 c_{\Sigma_*,p}=c_{O,p}=m_+ \qquad\text{for every } p\in(1,3).
\]
\end{theorem}

Theorem \ref{thm:main} may be viewed as an analogue of Bray's inequality
\eqref{eq:bray} and of its $p$-capacity extensions \cite{Xiao, XYZ,
MazurowskiYao}, requiring no horizon: the minimal boundary is replaced by
the infimum over outward-minimizing separating boundaries.

When the second end is also asymptotically flat, the equality horizon
compares the two masses.

\begin{corollary}[Two-end comparison and Schwarzschild rigidity]
\label{cor:mass-comparison}
Assume equality holds in \eqref{eq:main-inequality} for some
$p_0\in(1,3)$.  If $e_-$ is asymptotically flat with ADM mass $m_-$, then
\[
 m_-\ \ge\ m_+,
\]
and equality holds if and only if $(M,g)$ is the two-sided spatial
Schwarzschild manifold of mass $m_+$.
\end{corollary}

In particular, if the mass--capacity equality holds from both ends,
possibly at different exponents, then $(M,g)$ is the Schwarzschild
double.

The next theorem gives a sufficient condition for the existence
of a horizon.

\begin{theorem}[A capacity gap detects a horizon]\label{thm:gap}
Assume $R_g\ge0$ and $A(g)>0$.  If
\[
 c_{M,p}\ <\ c_{O,p}
\]
for some $p\in(1,3)$, then $(M,g)$ contains a horizon.
\end{theorem}

\begin{remark}
For $p=2$, the assumption $A(g)>0$ can be omitted if
$\operatorname{Ric}_g$ is bounded from below and $c_{M,2}>0$.
Indeed, let $u$ be the bounded harmonic function obtained by exhaustion
in \cite[Proposition 3.1]{MiaoJGA}.
The Cheng--Yau gradient estimate \cite{ChengYau} gives
$\|\nabla u\|_{L^\infty(M)}<\infty$.
Since $4\pi c_{M,2}\le \|\nabla u\|_{L^\infty(M)}|\Sigma|$
for every $\Sigma\in\mathcal S$, we obtain
\[
 A(g)\ge \frac{4\pi c_{M,2}}{\|\nabla u\|_{L^\infty(M)}}>0.
\]
\end{remark}

The paper is organized as follows.  Section \ref{sec:proofs} records the
normalized inequalities of \cite{XYZ} and proves the three main results.
Two elementary capacity lemmas are deferred to Appendix
\ref{app:lemmas}.

\subsection*{Acknowledgements}
The author thanks his advisor, Professor Pengzi Miao, for introducing
him to the notion of outward-minimizing capacity and for his guidance
and many valuable discussions.

\begingroup\emergencystretch=2em
This research was supported by Global - Learning \& Academic research
institution for Master's$\cdot$PhD students, and Postdocs (G-LAMP)
Program of the National Research Foundation of Korea (NRF) grant
funded by the Ministry of Education (No. RS-2025-25442252).
\par\endgroup

\section{Proofs of the main results}\label{sec:proofs}

\subsection{Normalized capacity inequalities}
\label{sec:prelim}

For a smooth surface $\Sigma\in\mathcal S$ set
\[
 \delta_\Sigma:=\Big(\frac1{16\pi}\int_\Sigma H^2\,d\sigma\Big)^{1/2},
 \qquad
 \kappa_\Sigma:=\frac{1-\delta_\Sigma}{1+\delta_\Sigma}
 \quad(\text{if }\delta_\Sigma<1),
\]
and suppose that $\delta_\Sigma<1$ and that the $e_+$-side
$M\setminus D_\Sigma$ is a one-end asymptotically flat manifold with
connected boundary $\Sigma$, $H_2(M\setminus D_\Sigma,\Sigma)=0$, and
$R_g\ge0$.  The identity
$\frac{4\kappa_\Sigma}{(1+\kappa_\Sigma)^2}=1-\delta_\Sigma^2$ shows
that $\kappa_\Sigma$ is precisely the parameter appearing in 
\cite{XYZ}.  Moreover,
$a(p-1)=3-p$ and
$2(I_a(\kappa_\Sigma)\mathfrak c_p)^{1/a}
=(I_a(\kappa_\Sigma)/I_a(1))^{1/a}c_{\Sigma,p}$, where
$\mathfrak c_p=(\pcap_p(\Sigma)/4\pi)^{1/(p-1)}$.  Thus, in the
normalization \eqref{eq:normalized-capacity}, the corresponding
inequalities of \cite{XYZ} read as follows for every $p\in(1,3)$:
\begin{align}
 m_+&\ \ge\
 \Big(\frac{I_a(\kappa_\Sigma)}{I_a(1)}\Big)^{1/a}c_{\Sigma,p},
 \label{eq:xyz-mass}\\
 \sqrt{\frac{|\Sigma|}{16\pi}}
 &\ \ge\ \frac{(1+\kappa_\Sigma)^2}{4\kappa_\Sigma}
 \Big(\frac{I_a(\kappa_\Sigma)}{I_a(1)}\Big)^{1/a}c_{\Sigma,p}.
 \label{eq:xyz-area}
\end{align}
If $\Sigma\in\mathcal S_O$, the
monotonicity of the Hawking mass along weak inverse mean curvature flow
\cite{HuiskenIlmanen} gives
\begin{equation}\label{eq:hawking-bound}
 m_+\ \ge\ m_H(\Sigma)\ =\ \sqrt{\frac{|\Sigma|}{16\pi}}\,
 (1-\delta_\Sigma^2).
\end{equation}
For $p=2$ one has $I_1(t)=\frac{t}{1+t}$, and \eqref{eq:xyz-mass},
\eqref{eq:xyz-area} reduce to Miao's mass-to-capacity ratio
\eqref{eq:miaoratio} and to the Bray--Miao capacity bound
\cite{BrayMiao}.

\subsection{Mass--capacity inequality and rigidity}

\begin{proof}[Proof of Theorem \ref{thm:main}]
Assume $A(g)>0$. Zhu's construction \cite{Zhu}, which implements
Gromov's $\mu$-bubble method \cite{GromovFour}, gives a sequence
$\{\Gamma_j\}\subset\mathcal S_O$ of strictly outward-minimizing $2$-spheres
with constant mean curvatures $\epsilon_j\to0$ and uniformly bounded
areas.  The $e_+$-side of each $\Gamma_j$ is a one-end asymptotically
flat manifold diffeomorphic to $\R^3$ minus a ball. In particular,
$\delta_{\Gamma_j}=\epsilon_j\sqrt{|\Gamma_j|/16\pi}\to0$ and
$\kappa_{\Gamma_j}\to1$.  Since $c_{O,p}\le c_{\Gamma_j,p}$,
\eqref{eq:xyz-mass} gives
\[
 m_+\ \ge\
 \Big(\frac{I_a(\kappa_{\Gamma_j})}{I_a(1)}\Big)^{1/a}
 c_{\Gamma_j,p}
 \ \ge\
 \Big(\frac{I_a(\kappa_{\Gamma_j})}{I_a(1)}\Big)^{1/a}c_{O,p}.
\]
Letting $j\to\infty$ proves \eqref{eq:main-inequality}.  We turn to
the equality case.

Suppose first that a horizon $\Sigma_*$ with the stated properties
exists.  Since the $e_+$-side of $\Sigma_*$ is the Schwarzschild
exterior of mass $m_+$, the Schwarzschild normalization 
\eqref{eq:normalized-capacity} gives $c_{\Sigma_*,p}=m_+$ for every
 $p$; since every $\Gamma\in\mathcal S_O$ encloses $\Sigma_*$,
 monotonicity \eqref{eq:capacity-monotonicity} gives
 $c_{\Sigma_*,p}\le c_{\Gamma,p}$.  Taking the infimum over
$\Gamma\in\mathcal S_O$ gives
$c_{O,p}=c_{\Sigma_*,p}=m_+$ for every $p$, and in particular equality
holds in \eqref{eq:main-inequality}.

Conversely, suppose $m_+=c_{O,p_0}$ for some $p_0\in(1,3)$.  Use Zhu's
nested rigidity construction \cite[Section 3]{Zhu}, and denote its
nested rigidity construction \cite{Zhu}, and denote its
strictly outward-minimizing CMC spheres by
$\{\Lambda_j\}\subset\mathcal S_O$,
with $D_{\Lambda_{j+1}}\subset D_{\Lambda_j}$ and mean curvatures
$\eta_j\to0$.  Their areas and their $p_0$-capacities are
non-increasing, so the limits
\[
 a_\infty:=\lim_{j\to\infty}|\Lambda_j|,
 \qquad
 c_\infty:=\lim_{j\to\infty}c_{\Lambda_j,p_0}
\]
exist.  Since the areas are bounded and $\eta_j\to0$, the corresponding
$\delta_{\Lambda_j}\to0$ and $\kappa_{\Lambda_j}\to1$.  The area
estimate \eqref{eq:xyz-area} and the Hawking bound
\eqref{eq:hawking-bound} give
\[
 c_{O,p_0}\ \le\ c_\infty\ \le\
 \sqrt{\frac{a_\infty}{16\pi}}\ \le\ m_+.
\]
Thus all inequalities above are equalities:
\begin{equation}\label{eq:pinch}
 c_\infty=\sqrt{\frac{a_\infty}{16\pi}}=m_+,
 \qquad a_\infty=16\pi m_+^2.
\end{equation}

By \eqref{eq:pinch}, $m_+^2=a_\infty/(16\pi)$.  The Hawking bound
\eqref{eq:hawking-bound} then gives
$|\Lambda_j|\le a_\infty+O(\eta_j^2)$, while nestedness gives the
exact lower bound $|\Lambda_k|\ge a_\infty$ for every later member of
 the chain.  These are precisely the area inputs in Zhu's rigidity
 argument \cite{Zhu}.  Zhu's coarea estimate forces the spheres to meet a
 fixed compact set.  Curvature estimates for stable $\mu$-bubbles due to
 Zhou--Zhu \cite{ZhouZhu} give local smooth convergence, while a result
 of Gromov--Lawson \cite{GromovLawson} rules out a noncompact finite-area
limit.  Hence a subsequence converges to an outward minimizing minimal
$2$-sphere $\Sigma_*\in\mathcal S_O$ with
$|\Sigma_*|=a_\infty=16\pi m_+^2$.

Since $\Sigma_*$ is outward minimizing and minimal and
$m_+=\sqrt{|\Sigma_*|/16\pi}$, the equality case of the Riemannian
Penrose inequality \cite{HuiskenIlmanen} shows that its $e_+$-side is
the Schwarzschild exterior of mass $m_+$.  By the Schwarzschild normalization in 
\eqref{eq:normalized-capacity}, $c_{\Sigma_*,p}=m_+$ for every
$p\in(1,3)$; in particular $\Sigma_*$ realizes the infimum
$c_{O,p_0}$.

\emph{Every $\Gamma\in\mathcal S_O$ encloses $\Sigma_*$.}
Let $\Gamma\in\mathcal S_O$ and suppose
$\operatorname{Vol}_g(D_{\Sigma_*}\setminus D_\Gamma)>0$.  Set
$E:=D_{\Sigma_*}\cap D_\Gamma$ and $\Sigma_E:=\partial E$.  Then
$\Sigma_E\in\mathcal S$, and Lemma \ref{lem:hull} shows that
$\Sigma_E\in\mathcal S_O$.

Choose a smooth separating region $D_0\subset E$, and let $H$ be a
least-perimeter enclosure of $D_0$.
Submodularity and the outward minimality of $E$ give
$P(H\cap E)\le P(H)$, so $F:=H\cap E$ is also a least-perimeter
enclosure of $D_0$.  Hence $\partial F\in\mathcal S_O$ is of class
$C^{1,1}$ by \cite{HuiskenIlmanen}.
Since $F\subset E$ and
$\operatorname{Vol}_g(D_{\Sigma_*}\setminus F)>0$, Lemma
\ref{lem:strict} gives
\[
 c_{O,p_0}\ \le\ c_{\partial F,p_0}\ <\ c_{\Sigma_*,p_0}\ =\ c_{O,p_0},
\]
a contradiction.  Hence every $\Gamma\in\mathcal S_O$ encloses
$\Sigma_*$.

\emph{Least area and equality at every exponent.}
For arbitrary $\Sigma\in\mathcal S$, its outward minimizing hull
$(D_\Sigma)^\sharp$ satisfies
$P((D_\Sigma)^\sharp)\le|\Sigma|$ and
$\partial((D_\Sigma)^\sharp)\in\mathcal S_O$.
The preceding enclosure conclusion therefore shows that
$D_{\Sigma_*}\subset(D_\Sigma)^\sharp$.
Since $\Sigma_*$ is outward minimizing,
$|\Sigma_*|\le P((D_\Sigma)^\sharp)$.
Thus
\[
 |\Sigma_*|\ \le\ P((D_\Sigma)^\sharp)\ \le\ |\Sigma|
 \qquad\text{for every }\Sigma\in\mathcal S.
\]
Taking the infimum yields $|\Sigma_*|\le A(g)$, whereas
$\Sigma_*\in\mathcal S$ gives $A(g)\le|\Sigma_*|$.  Consequently
$|\Sigma_*|=A(g)=16\pi m_+^2$.  Finally, for every $p\in(1,3)$ and every
$\Gamma\in\mathcal S_O$, the enclosure conclusion and
\eqref{eq:capacity-monotonicity} give
$c_{\Sigma_*,p}\le c_{\Gamma,p}$, while the Schwarzschild normalization 
gives $c_{\Sigma_*,p}=m_+$; hence $c_{O,p}=c_{\Sigma_*,p}=m_+$ for
every $p$.
\end{proof}

\begin{proof}[Proof of Corollary \ref{cor:mass-comparison}]
Zhu's mass--systole inequality \cite{Zhu}, applied with
$e_-$ as the distinguished end, gives $m_-\ge\sqrt{A(g)/16\pi}$, and
Theorem \ref{thm:main} gives $A(g)=16\pi m_+^2$; hence $m_-\ge m_+$.

Suppose $m_-=m_+$.  Then equality holds in Zhu's inequality for $e_-$,
and Zhu's rigidity result \cite{Zhu} produces a horizon
$\widetilde\Sigma\in\mathcal S$ of area $A(g)$ that is strictly outward
minimizing toward $e_-$ and whose $e_-$-side is the Schwarzschild
exterior of mass $m_-$.  Toward $e_+$ we also have
$\widetilde\Sigma\in\mathcal S_O$.  Indeed, if
$F\supset D_{\widetilde\Sigma}$ is any finite-perimeter competitor, then
$\partial F\in\mathcal S$.  Hence \eqref{eq:area-infimum} gives directly
\[
 P(F)\ \ge\ A(g)\ =\ |\widetilde\Sigma|.
\]

The minimal surface $\widetilde\Sigma$ cannot enter the $e_+$-side
Schwarzschild exterior $\Omega$ of $\Sigma_*$: the region $\Omega$ is foliated by
strictly mean-convex coordinate spheres, and sliding the foliation
inward to a first touching point with $\widetilde\Sigma$ would violate
the comparison principle for the mean curvature.  Hence
$\widetilde\Sigma$ is enclosed by $\Sigma_*$.  If
$\widetilde\Sigma\cap\Sigma_*=\emptyset$, then the outward minimizing
surface $\widetilde\Sigma$ is strictly enclosed by $\Sigma_*$, and
Lemma \ref{lem:strict} gives
\[
 c_{O,p_0}\le c_{\widetilde\Sigma,p_0}
 <c_{\Sigma_*,p_0}=c_{O,p_0},
\]
a contradiction.  Hence $\widetilde\Sigma$ and $\Sigma_*$
touch; as both are minimal and $\widetilde\Sigma$ lies on one side of
$\Sigma_*$, the strong maximum principle gives
$\widetilde\Sigma=\Sigma_*$.  Therefore $(M,g)$ is the union of two
Schwarzschild exteriors of mass $m_+$ glued along their common horizon,
i.e.\ the two-sided spatial Schwarzschild manifold.  The converse is
immediate from the explicit computation on the Schwarzschild double.
\end{proof}

\subsection{Capacity gap and horizons}\label{sec:horizon}

\begin{proof}[Proof of Theorem \ref{thm:gap}]
We argue by contraposition: assuming $(M,g)$ contains no horizon, we
show $c_{M,p}=c_{O,p}$ for every $p\in(1,3)$.

Let $\{\Gamma_j\}\subset\mathcal S_O$ be the outward minimizing
sequence constructed by Zhu \cite{Zhu} and used above.  By the same
compactness argument as in the rigidity proof, a
subsequence meeting a fixed compact set would produce a horizon.
The uniform area bound rules out escape through the asymptotically flat
end $e_+$.  Therefore the spheres leave every compact set through $e_-$.

Fix $\Sigma\in\mathcal S$ and $p\in(1,3)$.  Choose an end-neighborhood
$U_-\subset D_\Sigma$ of $e_-$.  The preceding escape conclusion gives
$\Gamma_j\subset U_-$ for all large $j$.  If $f$ is admissible for
$\Sigma$, it vanishes on $U_-$; replacing it by zero on $D_{\Gamma_j}$
produces an admissible function for $\Gamma_j$ without increasing its
energy.  Therefore
\[
 c_{O,p}\ \le\ c_{\Gamma_j,p}\ \le\ c_{\Sigma,p}.
\]
Taking the infimum over $\Sigma\in\mathcal S$ yields
$c_{O,p}\le c_{M,p}$; the reverse inequality is trivial.
\end{proof}

\appendix

\section{Two capacity lemmas}\label{app:lemmas}

We record the two elementary facts used in the proof of Theorem
\ref{thm:main} in Section \ref{sec:proofs}.

\begin{lemma}[Strict monotonicity]\label{lem:strict}
Let $\Sigma_1\in\mathcal S_O$ be of class $C^{1,1}$,
and let $\Sigma_2\in\mathcal S$ satisfy
$D_{\Sigma_1}\subset D_{\Sigma_2}$ and
$\operatorname{Vol}_g(D_{\Sigma_2}\setminus D_{\Sigma_1})>0$.  Then
$c_{\Sigma_1,p}<c_{\Sigma_2,p}$ for every $p\in(1,3)$.
\end{lemma}

\begin{proof}
Take the closed representative of $D_{\Sigma_1}$ with $C^{1,1}$
boundary.  Its open exterior has no bounded connected component:
filling such a component would remove its boundary and strictly
decrease perimeter, contrary to outward minimality.
Since the exterior contains only the end $e_+$, it is connected.
Let $u_i$ be the $p$-capacitary potential associated
with $\Sigma_i$.  The minimizer is unique by strict convexity of
$\xi\mapsto|\xi|^p$ for $p>1$, and $u_1$ is positive on the exterior of
$D_{\Sigma_1}$ by the strong maximum principle for $p$-harmonic
functions \cite{HKM}.  Since $D_{\Sigma_1}\subset D_{\Sigma_2}$, the
function $u_2$ is an admissible competitor for the $\Sigma_1$-problem,
with $p$-energy $\pcap_p(\Sigma_2)$.  It vanishes on the positive-volume
set $D_{\Sigma_2}\setminus D_{\Sigma_1}$, which lies in the open
exterior up to a null set, and hence differs from $u_1$.
Uniqueness of the minimizer gives
$\pcap_p(\Sigma_1)<\pcap_p(\Sigma_2)$, and
\eqref{eq:normalized-capacity} converts this into the assertion.
\end{proof}

\begin{lemma}[Intersection of outward-minimizing boundaries]\label{lem:hull}
If $\Sigma_1,\Sigma_2\in\mathcal S_O$, then
\[
 \partial(D_{\Sigma_1}\cap D_{\Sigma_2})\in\mathcal S_O.
\]
\end{lemma}

\begin{proof}
Set $D_i:=D_{\Sigma_i}$ and $E:=D_1\cap D_2$.  The region $E$ has finite
perimeter, contains an end-neighborhood of $e_-$, and its complement
contains an end-neighborhood of $e_+$; hence $\partial E\in\mathcal S$.
Let $F\supset E$ be a finite-perimeter region with $F\setminus E\Subset M$.
Submodularity of the perimeter \cite{Maggi} gives
\[
 P(D_1\cap F)+P(D_1\cup F)\le P(D_1)+P(F).
\]
Moreover,
$(D_1\cup F)\setminus D_1\subset
F\setminus(D_1\cap D_2)\Subset M$.
Since $D_1$ is outward minimizing,
$P(D_1)\le P(D_1\cup F)$, and therefore
\begin{equation}\label{eq:first-intersection}
 P(D_1\cap F)\le P(F).
\end{equation}
Applying submodularity to $D_2$ and $D_1\cap F$ gives
\[
 P(D_1\cap D_2)+P(D_2\cup(D_1\cap F))
 \le P(D_2)+P(D_1\cap F),
\]
where we used $F\supset D_1\cap D_2$.  Also,
\[
 [D_2\cup(D_1\cap F)]\setminus D_2
 \subset F\setminus(D_1\cap D_2)\Subset M.
\]
The outward minimizing property of $D_2$ gives
$P(D_2)\le P(D_2\cup(D_1\cap F))$.  Hence, by
\eqref{eq:first-intersection},
\[
 P(D_1\cap D_2)\le P(D_1\cap F)\le P(F),
\]
as required.
\end{proof}

\end{document}